\documentclass[12pt,reqno]{amsart}
\usepackage{amsmath} 
\usepackage{amssymb}
\usepackage{dsfont}
\usepackage[dvips,draft,final]{graphics}
\usepackage{amssymb,amsmath}
\usepackage[T1]{fontenc}
\usepackage{fancyhdr}
\usepackage{url}

\usepackage{color}
\usepackage{graphicx}

\def\squarebox#1{\hbox to #1{\hfill\vbox to #1{\vfill}}}

\newtheorem{Thm}{Theorem}[section]
 \newtheorem{cor}{Corollary}[section]
\newtheorem{Def}{Definition}[section]
\newtheorem{lem}{Lemma}[section]

\numberwithin{equation}{section}

\newcommand{\bel}{\begin{equation} \label}
\newcommand{\ee}{\end{equation}}

\newcommand{\re}{\mathfrak R}
\newcommand{\im}{\mathfrak I}

\newcommand{\R}{\mathbb{R}}

\def\epsilon{\varepsilon}
\def\phi {\varphi}

\newtheorem{rem}{Remark}[section]
\newtheorem{prop}{Proposition}[section]

\providecommand{\abs}[1]{\left\lvert#1\right\rvert}
\providecommand{\norm}[1]{\left\lVert#1\right\rVert}
\numberwithin{equation}{section}

\renewcommand{\leq}{\leqslant}
\renewcommand{\geq}{\geqslant}
\providecommand{\abs}[1]{\left\lvert#1\right\rvert}
\providecommand{\norm}[1]{\left\lVert#1\right\rVert}

\def\beq{\begin{equation}}
\def\eeq{\end{equation}}
\newcommand{\bea}{\begin{eqnarray}}
\newcommand{\eea}{\end{eqnarray}}
\newcommand{\beas}{\begin{eqnarray*}}
\newcommand{\eeas}{\end{eqnarray*}}

{

\title[Determination of the sound speed and an initial source]{Determination of the sound speed
and an initial source in photoacoustic tomography}

\author{Yavar Kian}
\address{Aix Marseille Univ, Universit\'e de Toulon, CNRS, CPT, Marseille, France.}
\email{yavar.kian@univ-amu.fr}

\author{Gunther Uhlmann}
\address{G. Uhlmann, Department of Mathematics\\
       University of Washington\\
       Seattle, WA  98195-4350\\
       USA}
\email{gunther@math.washington.edu}

\begin{document}

\maketitle

\begin{abstract} In thermoacoustic and photoacoustic tomography an electromagnetic wave is sent through a medium, heating it and therefore generating an elastic expansion that in turns generates an acoustic wave that is measured outside the medium. The general problem is to recover both the inhomogeneous sound speed and the initial pressure from the boundary measurements of the solution of the acoustic wave equation  with a single measurement of the pressure. We show that one can recover both the sound speed and the initial pressure assuming a monocitiy condition on the sound speed that includes the physically interesting case of  piecewise constant sound speeds.  We also establish a link between this problem and the corresponding transmission eigenvalue problem.\\


\end{abstract}

\section{Introduction}
\label{sec-intro}
\setcounter{equation}{0}

In photoacoustic and thermoacoustic tomography we probe a medium with an electromagnetic wave. This excites the medium and produces heat generating an elastic expansion which in turn generates a corresponding sound wave that is measured at the boundary of the medium. In photoacoustic tomography (PAT) this is done with a rapidly pulsating laser beam; in thermoacoustic tomography (TAT) the medium is probed with an electromagnetic wave of a lower frequency (see e.g. \cite{KRK,W}). This imaging method is also referred to as the ``sound of light". In many practical situations the sound speed inside the medium is  unknown and in order to determine this sound speed additional measurement have been considered in \cite{JW}. This is modeled as an inverse problem for the acoustic wave equation,

\begin{equation}\label{intro-eqn}\left\{\begin{array}{ll}c^{-2}(x)\partial_t^2u+\Delta u=0,\quad &\textrm{in}\ \R_+\times\R^3,\\  u(0,x)=f(x),\quad \partial_tu(0,x)=0,&x\in\R^3,\end{array}\right.\end{equation}
with $f\in H^1(\R^3)$ compactly supported and $c\in L^\infty(\R^3)$ taking a fix  constant value outside a compact set and satisfying the following condition
 \begin{equation}\label{c}c(x)\geq r_1>0,\quad x\in\R^3,\end{equation}
where the initial condition $f$ models the initial pressure and $c$ is the sound speed of the medium (see e.g. \cite{DSK,T}). The problem of recovering the initial pressure if the sound speed is known has been intensively studied and without being exhaustive one can refer to \cite{AKK,SU1,SU2,SU4} (see also \cite{HK} for similar results for Lam\'e systems). Nevertheless, there have been only few works devoted to the simultaneous determination of the sound speed coefficient $c$ and the initial pressure $f$ which is proved to be an unstable inverse problem \cite{SU3}. For the simultaneous determination of the sound speed coefficient $c$ and the initial pressure $f$ we are only aware of the works  \cite{KM,LU}. In \cite{LU} the authors proved the unique determination of $c^{-2}f$ when the unknown part of this function is harmonic or independent of at least one spatial variable. When the sound speed coefficient $c$ is constant and $f$ non-negative \cite{LU}  proved the unique simultaneous determination of $c$  and $f$. The work of \cite{LU} has been extended by \cite{KM} who have proved the unique simultaneous determination of $c$  and $f$ provided that the condition
\bel{kmm}\int_{\R^3}c^{-2}fdx\neq0\ee
is fulfilled and the unknown part of $c^{-2}$ is harmonic. As far as we know the simultaneous determination of the sound speed coefficient $c$ and the initial pressure $f$ when $c$ is not constant or condition \eqref{kmm} is not fulfilled remain an open problem and the goal of this article is to study this problem.

\section{Statement of the results}
Now we consider a precise mathematical formulation of the problem when the Laplacian $\Delta$ in \eqref{intro-eqn} is replaced by a general second order operator $\mathcal{A}$ in divergence form. Let $a:=(a_{i,j})_{1 \leq i,j \leq 3} \in C^1(\R^3;\R^{3^2})$, be symmetric, i.e. such that 
$$ a_{i,j}(x)=a_{j,i}(x),\ x \in \Omega,\ i,j = 1,2,3, $$
and satisfying the ellipticity condition
\bel{a1}
\exists b>0,\ \sum_{i,j=1}^3 a_{i,j}(x) \xi_i \xi_j \geq b |\xi|^2,\ x \in \R^3,\ \xi=(\xi_1,\xi_2,\xi_3) \in \R^3.
\ee
We assume also that there exists $R_1>0$ such that
\bel{a2} a_{i,j}(x)=\delta_{ij},\quad x\in\R^3,\ |x|>R_1,\ i,j=1,2,3,\ee
where $\delta_{ij}$ denotes the Kronecker delta symbol. We define the operator $\mathcal A$ by
$$ 
\mathcal A v(x) :=-\sum_{i,j=1}^3 \partial_{x_i} 
\left( a_{i,j}(x) \partial_{x_j} v(x) \right),\  x\in\R^3. 
$$ 
Then, we consider the initial value problem (IVP)

\begin{equation}\label{eq1}\left\{\begin{array}{ll}c^{-2}(x)\partial_t^2u+\mathcal A u=0,\quad &\textrm{in}\ \R_+\times\R^3,\\  u(0,x)=f(x),\quad \partial_tu(0,x)=0,&x\in\R^3,\end{array}\right.\end{equation}
with $f\in H^1(\R^3)$ compactly supported and $c\in L^\infty(\R^3)$ taking a fix  constant value outside a compact set and satisfying \eqref{c}.

We fix $\Omega$ a Lipschitz bounded and  connected domain  of $\R^3$ such that supp$(f)\subset\overline{\Omega}$ and with $c$ constant on $\R^3\setminus\overline{\Omega}$. We study the inverse problem of determining simultaneously the sound speed coefficient $c$ and the initial pressure $f$ from the knowledge of $u(t,x)$, $(t,x)\in\R_+\times\partial\Omega$.

In order to state our  main results dealing with the simultaneous determination of the sound speed coefficient $c$ and the initial pressure $f$, we first need to define some class of admissible sound speed coefficients. For this purpose, let us consider $c\in L^\infty(\R^3)$ satisfying  condition \eqref{c} and  such that $c$ is constant on the set $\R^3\setminus \overline{\Omega}$. We fix also $R_0>0$ such that  $\overline{\Omega}\subset B_{R_0}:=\{x\in\R^3:|x|<R_0\}$ and a constant $c_0>0$ such that $c=c_0$ on $\R^3\setminus\overline{\Omega}$. Then, we define the class of admissible sound speed coefficients $c$ as follows.
\begin{Def}\label{d1} The sound speed coefficient $c$ will be called admissible if condition \eqref{c} is fulfilled and if for any $f\in H^1(\R^3)$, $g\in L^2(\R^3)$ with supp$(f)\cup$supp$(g)\subset\overline{\Omega}$, and any $R\geq R_0$, there exists a constant $C>0$ and a constant $\delta>0$ depending on $R$, $a$, $c$ and $\Omega$ such that the solution $w$ of the IVP
\bel{tata}\left\{\begin{array}{ll}c^{-2}(x)\partial_t^2w+\mathcal A w=0,\quad &\textrm{in}\ \R_+\times\R^3,\\  w(0,x)=f(x),\quad \partial_tw(0,x)=g(x),&x\in\R^3\end{array}\right.\ee
satisfies the estimate
\begin{equation}\label{d1a}\norm{(w(t,\cdot),\partial_tw(t,\cdot))}_{H^1(B_{R})\times L^2(B_{R})}\leq Ce^{-\delta t}(\norm{f}_{H^1(\R^3)}+\norm{g}_{L^2(\R^3)}),\quad t>0.\end{equation}
\end{Def}

Recall that estimate \eqref{d1a} is called the local energy decay for the wave equation and, for smooth sound speed coefficients $c$ and $a$, it is known to be a consequence of the non-trapping condition imposed on $c$ and $a$ (see e.g. \cite{Va}) considered also in \cite{LU}.  We mention also the work of \cite{Vo2} where an equivalent condition to \eqref{d1a} has been obtained for smooth coefficients $c$ and $a$ by applying the analysis of \cite{Vo1}.

Our first main result can be stated as follows.
\begin{Thm}\label{t3} 
For $j=1,2$, let $f_j\in H^1(\R^3)$ be non-uniformly  vanishing and $c_j\in L^\infty(\R^3)$ be  an admissible sound speed coefficient, such that supp$(f_j)\cup$supp$(c_1-c_2)\subset\overline{\Omega}$ and $c_j$ is constant on  $\R^3\setminus\overline{\Omega}$. Let $u_j$ be the solution of \eqref{eq1} with $c=c_j$ and $f=f_j$. Consider $K_1,\ldots,K_N$, $N$ compact  sets, with no-empty interior, included in $\overline{\Omega}$ 
such that 
\begin{equation}\label{t3a}  K_i\cap K_j=\emptyset,\quad i\neq j,\end{equation}
\begin{equation}\label{t3b} \mathcal O =B_{R_0}\setminus\left(\bigcup_{j=1}^N K_j\right)\ \ \textrm{is connected}.\end{equation}
 Now assume that $c_1-c_2\in C(\overline{\Omega})$ and that  the following conditions
\begin{equation}\label{t3c}  c_1(x)= c_2(x),\quad x\in\mathcal O,\end{equation}
\begin{equation}\label{t3d} \forall j=1,\ldots,N,\quad  c_1|_{K_j}\leq c_2|_{K_j}\quad\textrm{or} \quad c_2|_{K_j}\leq c_1|_{K_j},\end{equation}
are fulfilled.
Then  the  implication \eqref{t1c} holds true.
\end{Thm}

As a corollary of this main result we consider a global monotonicity condition \eqref{t1b} in lieu of the local one used above. This corollary can be stated as follows.

\begin{cor}\label{t1} 
For $j=1,2$, let $f_j\in H^1(\R^3)$ be non-uniformly vanishing and let $c_j\in L^\infty(\R^3)$ be  an admissible sound speed coefficient, such that supp$(f_j)\cup$supp$(c_1-c_2)\subset\overline{\Omega}$ and $c_j$ is constant on  $\R^3\setminus\overline{\Omega}$. Let $u_j$ be the solution of \eqref{eq1} with $c=c_j$ and $f=f_j$. Assume that $c_1-c_2\in C(\overline{\Omega})$ and that  the following condition
\begin{equation}\label{t1b}  c_1\leq c_2\quad\textrm{or} \quad c_2\leq c_1\end{equation}
is fulfilled.
Then  the following implication
\begin{equation}\label{t1c}  (u_1= u_2 \textrm{ on }\R_+\times\partial\Omega)\Rightarrow (c_1=c_2 \textrm{ and }f_1=f_2)\end{equation}
holds true.
\end{cor}

Let us consider the set of data satisfying,
$$\int_0^{+\infty} t^{2n}u(t,x)dt,\quad x\in\partial\Omega,\quad n\in\mathbb N.$$
By imposing some conditions on this set of data, we obtain the following improvement of Corollary \ref{t1} 

\begin{Thm}\label{t5} 
For $j=1,2$, let $f_j\in H^1(\R^3)$ be non-uniformly  vanishing and real valued. Let $c_j\in L^\infty(\R^3)$ be  an admissible sound speed coefficient, such that supp$(f_j)\cup$supp$(c_1-c_2)\subset\overline{\Omega}$ and $c_j$ is constant on  $\R^3\setminus\overline{\Omega}$. Let $u_j$ be the solution of \eqref{eq1} with  $c=c_j$, $f=f_j$ and consider the set
\bel{t5a}\mathcal B:=\left\{n\in\mathbb N:\ \int_0^{+\infty}t^{2n}u_1(t,\cdot)|_{\partial\Omega}dt\not\equiv 0\right\}.\ee
We define $k_0=\inf \mathcal B$, with $k_0=\infty$ if $\mathcal B=\emptyset$, and we assume that if $k_0<\infty$ the function $g\in H^{\frac{1}{2}}(\partial\Omega)$ defined by 
\bel{g}g:=\int_0^{+\infty}t^{2k_0}u_1(t,\cdot)|_{\partial\Omega}dt\ee
is of constant sign. We assume also that  
\begin{equation}\label{t5b}  c_1(x)- c_2(x)=\psi(x)h(x),\quad x\in\Omega\end{equation}
with $h\in C(\overline{\Omega})$ of constant sign and $\psi\in H^1(\Omega)\cap L^\infty(\Omega)$ satisfying $\mathcal A\psi=0$ on $\Omega$.
Then  the  implication \eqref{t1c} holds true.
\end{Thm}

In the same way as in Theorem \ref{t3}, we can also prove the following extension of Theorem \ref{t5}.
\begin{cor}\label{c1} 
Let the condition of Corollary \ref{t1} be fulfilled. Let $u_j$ be the solution of \eqref{eq1} with $c=c_j$ and $f=f_j$. Consider $K_1,\ldots,K_N$, $N$ compact sets, with no-empty interior, included in $\overline{\Omega}$ 
such that \eqref{t3a}-\eqref{t3c} are fulfilled and assume that
\begin{equation}\label{c1a} \forall j=1,\ldots,N,\quad  c_1(x)-c_2(x)=\psi_j(x)h_j(x),\quad x\in K_j,\end{equation}
with $h_j\in C_0(K_j)$ of constant sign and $\psi_j\in H^1(\Omega)\cap L^\infty(\Omega)$ satisfying $\mathcal A\psi_j=0$ on $\Omega$.
Assume also that if $\mathcal B\neq\emptyset$ and $k_0=\min \mathcal B$, the function $g$ given by \eqref{g} is of constant sign.
Then  the  implication \eqref{t1c} holds true.
\end{cor}

The sign condition imposed on the Dirichlet data \eqref{g}, under consideration in Theorem \ref{t5} and Corollary \ref{c1}, depends on the sound speed coefficient $c$ and the initial pressure $f$. We give below general  context where this condition is fulfilled.

\begin{prop}\label{r1} Let  $f_1\in H^1(\R^3)$ be non-uniformly  vanishing and real valued. Let $c_1\in L^\infty(\R^3)$ be  an admissible sound speed coefficient, such that supp$(f_1)\subset\overline{\Omega}$ and $c_1=c_0$  on  $\R^3\setminus\overline{\Omega}$ with $c_0>0$ a constant. Assuming that $\mathcal A=-\Delta$ and condition \eqref{kmm} is fulfilled with $f=f_1$ and $c=c_1$, we will have $\min\mathcal B=1$ and the function $g$ given by \eqref{g} is a non-vanishing constant function. In the same way, if \eqref{kmm} is not fulfilled but the function
\bel{h}h(x):=\int_{\R^3}|x-y|^2\frac{f_1(y)}{c^{2}_1(y)}dy+\frac{c_0^2}{4\pi}\int_{\R^3}\int_{\R^3}\left(c_0^{-2}-c_1(y)^{-2}\right)\frac{f_1(y_1)}{|y-y_1|c^{2}_1(y_1)}dydy_1,\quad x\in\partial\Omega\ee
is non-uniformly vanishing and of constant sign, we will have $\min\mathcal B=2$ and the function $g$ given by \eqref{g} is  of constant sign. Therefore, in view of Theorem \ref{t5} and Corollary \ref{c1}, in these two situations by assuming that one of the conditions \eqref{t5b} and \eqref{c1a} is fulfilled  the implication \eqref{t1c} will hold true.\end{prop}

Let us observe that in Theorem \ref{t3} and Corollary \ref{t3} we prove the simultaneous unique determination of the sound speed coefficient $c$ and the initial pressure $f$ under the monotonicity condition \eqref{t1b} and the admissibility condition of Definition \ref{d1}. As mentioned above, the only other results that we are aware of about the TAT and PAT problem with unknown sound speed coefficient $c$ and initial pressure $f$ can be   found in \cite{KM,LU}. In contrast to \cite{LU} we can show the simultaneous determination of the unknown parameters $c$ and $f$ when $c$ is variable and in contrast to \cite{KM} we only assume that the initial pressure $f$ is non-uniformly vanishing. Moreover, our result is stated in an inhomogeneous medium described by a general  elliptic operator $\mathcal A$ with variable coefficients while the results of \cite{LU,KM} are restricted to $\mathcal A=-\Delta_x$. While the result of Corollary \ref{t1} is subjected to the global monotonicity condition \eqref{t1b}, we prove in Theorem  \ref{t3} that under some suitable assumptions this condition can be replaced by the local monotonicity condition \eqref{t3d}. The result of Theorem \ref{t3} can for instance be applied to the determination of a piecewise constant (or a smooth counterpart) sound speed coefficient $c$ provided that condition \eqref{d1a} is fulfilled. For this reason we believe that Theorem \ref{t3} can be useful in several practical situations.

By assuming that the Dirichlet trace $g$, given by \eqref{g}, is of constant sign, we improve in Theorem \ref{t5}  the results of Corollary \ref{t1}  by showing the unique determination of sound speed coefficients whose unknown part is a product of a function of constant sign with a harmonic function. In the spirit of Theorem \ref{t3}, we  extend in Corollary \ref{c1} the result of Theorem \ref{t5} to give the determination of sound speed coefficients whose unknown part is locally the product of a function of constant sign with a harmonic function. In that sense, Corollary \ref{c1} can for instance be applied to the determination of a piecewise harmonic (or smooth) sound speed coefficient $c$ provided that condition \eqref{d1a} is fulfilled. Note that the main assumption of Theorem \ref{t3} and Corollary \ref{c1} is imposed on the measurement under consideration given by the sign condition of \eqref{g} which is defined by the data $u(t,x)$, $(t,x)\in\R_+\times\partial\Omega$, with $u$ the solution of \eqref{eq1}. This condition depends on the sound speed coefficient $c$ and the initial pressure $f$ and we give in Proposition \ref{r1} some general examples of situation where this condition is fulfilled  with explicit conditions imposed to  $c_1$ and $f_1$. Namely, we show in Proposition \ref{r1} that the sign condition imposed on the Dirichlet trace $g$ in Theorem \ref{t5} and Corollary \ref{c1} will be fulfilled if $\mathcal A=-\Delta$ and \eqref{kmm} is fulfilled with $c=c_1$ and $f=f_1$. We prove also that when \eqref{kmm} is not fulfilled, a suitable sign condition imposed on the function \eqref{h} implies the sign condition imposed on the Dirichlet trace $g$ for Theorem \ref{t5} and Corollary \ref{c1}.  In that sense the condition of Theorem \ref{t5} and Corollary \ref{c1} extend the condition \eqref{kmm} under consideration in  \cite{KM} and we give an explicit condition  allowing Theorem \ref{t5} and Corollary \ref{c1} to hold true when \eqref{kmm} is not fulfilled.

Let us remark that our analysis differs from the one of \cite{KM,LU} mainly based on applications of results of scattering theory. Here our analysis is mostly based on an explicit asymptotic expansion of the Laplace transform in time of the solution of \eqref{eq1} for small frequency parameter. This approach allows us to consider a general elliptic differential operator not restricted to the Laplacian under consideration in \cite{KM,LU}. Moreover, in contrast to the analysis of \cite{KM,LU}, where such asymptotic expansions have been considered to order two or three, in this article we consider the complete asymptotic expansion. Using such property we can show the determination of the sound speed coefficient $c$ independently of the initial pressure $f$. In our analysis, we use also some long time asymptotic property of the solution of \eqref{eq1} combined with some properties of complex analysis.

In contrast to \cite{LU} we consider the Laplace transform of solutions of \eqref{eq1} and not their Fourier transform. This approach allows us to define the Laplace transform of the solutions of \eqref{eq1} as function taking values in $H^1(\R^3)$. Then, using the local energy decay \eqref{d1a} we can consider the holomorphic extension of the Laplace transform in time of the solution of \eqref{eq1} restricted to $x\in B_{R_0}$. 

Recall that, for smooth coefficients $a$ and $c$, the local energy decay \eqref{d1a} is a consequence of the non-trapping condition also considered in \cite{LU}. Moreover, it holds true for any odd dimension of space higher than $2$, hence our results hold true in any space $\R^n$ with $n\geq3$ odd. According to \cite{Vo2}, even for smooth coefficients $a$ and $c$, the non-trapping condition implies \eqref{d1a} but it is not equivalent in general to \eqref{d1a}. Moreover, it is not clear that the smoothness of the coefficients $a$ and $c$ are a necessary condition for proving \eqref{d1a}. For this reason, we define our class of admissible coefficients by mean of Definition \ref{d1} in order to prove our result with the weakest assumptions.

\ \\
This article is organized as follows. In Section 2, we recall some general properties of the Laplace transform in time of the solution of \eqref{eq1} as well the identities \eqref{l2c} which are the key ingredients of our proof. Then, in Section 3 we exploit all these results for proving Corollary \ref{t1} while Section 4 and 5 will be devoted to the proof of Theorem \ref{t3}, Theorem \ref{t5}, Corollary \ref{c1} and Proposition \ref{r1}. In Section 6,  we establish the connection between the transmission eigenvalues problem and the TAT/PAT problem. More precisely, we extend the analysis of \cite{HK} by giving in Theorem \ref{t4} an explicit link between the transmission eigenvalues and the TAT/PAT problem. Finally, in the appendix we recall and prove a result about the simultaneous determination of the initial state and the initial velocity of wave equations when the sound speed coefficient  $c$ is known but  non-smooth.

\section{Preliminary properties of solutions of \eqref{eq1} with admissible coefficient $c$}
From now on and in all the remaining parts of this article, for any open set $U$ of $\mathbb C$ and any Banach space $X$, we denote by $\mathcal H(U;X)$ the set of holomorphic functions on $U$ taking values in $X$.
Let us first recall that the solution of \eqref{eq1} is lying in $C^1([0,+\infty);L^2(\R^3))\cap C([0,+\infty);H^1(\R^3))$. Moreover one can verify (see e.g. \cite[Proposition 8]{HK}) that the Laplace transform  in time $\hat{u}$ of of the solution $u$ is well defined by 
$$\hat{u}(p,\cdot)=\int_0^{+\infty}e^{-pt}u(t,\cdot),\quad p\in\mathbb C_+:=\{z\in\mathbb C:\re z>0\}.$$
In addition, we have $p\mapsto \hat{u}(p,\cdot)\in \mathcal H(\mathbb C_+;H^1(\R^3))$ and applying the Laplace transform in time to \eqref{eq1} we deduce that
\begin{equation}\label{eq2}\mathcal A \hat{u}(p,x)+c^{-2}(x)p^2\hat{u}(p,x)=pc^{-2}(x)f(x),\quad x\in\R^3,\ p\in\mathbb C_+.\end{equation}

From now on and in all the remaining parts of this article, we denote by $\delta$ the constant appearing in \eqref{d1a} with $R=R_0$. For the proof of Corollary \ref{t1}, we will need first to recall some preliminary properties of the Laplace transform in time $\hat{u}(p,\cdot)$ of the solution $u$ of \eqref{eq1} when $c$ is admissible.  Namely, using the estimate \eqref{d1a} we will show that the map $p\mapsto \hat{u}(p,\cdot)|_{B_{R_0}}\in \mathcal H(\mathbb C_+;H^1(B_{R_0}))$ admits an analytic extension to $\mathbb C_{-\delta}:=\{z\in\mathbb C:\re z>-\delta\}$. Then, using this analytic extension we will derive the full asymptotic properties of $\hat{u}(p,\cdot)|_{B_{R_0}}$ as $p\to0$.

Assuming that $c$ is admissible and supp$(f)\subset\overline{\Omega}$ we can show the following first preliminary result.
\begin{lem}\label{l1} Assume that $c$ is admissible and supp$(f)\subset\overline{\Omega}$. Then, we can prove that $p\mapsto\hat{u}(p,\cdot)|_{B_{R_0}}$ admits an holomorphic extension to an element of 
$\mathcal H(\mathbb C_{-\delta};H^1(B_{R_0}))$. Moreover, fixing $(u^{(k)})_{k\geq0}$ a sequence in  $H^1(B_{R_0})$ defined by 
\bel{l1aa}u^{(k)}(x)=\frac{(-1)^k\int_0^{+\infty}t^ku(t,x)dt}{k!},\quad x\in B_{R_0},\ k\geq0,\ee
we have
\begin{equation}\label{l1a}\hat{u}(p,\cdot)|_{B_{R_0}}=\sum_{k=0}^\infty u^{(k)}p^k,\quad p\in\mathbb D_\delta:=\{z\in\mathbb C:|z|<\delta\}.\end{equation}
Finally, the functions $u^{(k)}$, $k\geq0$,  satisfy the following conditions
\begin{equation}\label{l1b}u^{(0)}= 0,\ \mathcal A u^{(1)}=c^{-2}f,\ \mathcal A u^{(k)}=-c^{-2}u^{(k-2)},\ k\geq 2,\ \textrm{in }B_{R_0},\end{equation}
\begin{equation}\label{l1c}\norm{u^{(k)}}_{H^1(B_{R_0})}\leq C\delta^{-k},\quad k\in\mathbb N,\end{equation}
with $C>0$ independent of $k$.
\end{lem}
\begin{proof}  We start by proving \eqref{l1c} and \eqref{l1a}. Recalling that $u$ satisfies the estimate \eqref{d1a}, one can check that $p\mapsto\hat{u}(p,\cdot)|_{B_{R_0}}\in \mathcal H(\mathbb C_{-\delta};H^1(B_{R_0}))$. Moreover, using estimate \eqref{d1a} and integrating by parts, we get 
$$\int_0^{+\infty}t^k\norm{u(t,\cdot)}_{H^1(B_{R_0})}dt\leq C\int_0^{+\infty}t^ke^{-\delta t}dt\leq C\delta^{-k}k!,\quad  k\geq0,$$
which implies \eqref{l1c}. Therefore, for all $p\in\mathbb D_\delta$, we have
$$\begin{aligned}\hat{u}(p,\cdot)|_{B_{R_0}}=\int_0^{+\infty}e^{-tp}u(t,\cdot)|_{B_{R_0}}dt&=\int_0^{+\infty}\sum_{k=0}^\infty \frac{(-tp)^k}{k!}u(t,\cdot)|_{B_{R_0}}dt\\
\ &=\sum_{k=0}^\infty \left((-1)^k\int_0^{+\infty} t^ku(t,\cdot)|_{B_{R_0}}dt\right)\frac{p^k}{k!},\end{aligned}$$
which clearly implies \eqref{l1a}.

Now let us consider \eqref{l1b}. Let us first prove that $u^{(0)}\equiv 0$. For this purpose, let us consider the solution $v$ of the IVP
$$\left\{\begin{array}{ll}c^{-2}(x)\partial_t^2v+\mathcal A v=0,\quad &\textrm{in}\ \R_+\times\R^3,\\  v(0,x)=0,\quad \partial_tv(0,x)=f(x),&x\in\R^3.\end{array}\right.$$
One can easily check that $u=\partial_tv$ and, in view of \eqref{d1a} we have 
$$\lim_{t\to+\infty}\norm{ v(t,\cdot)}_{H^1(B_{R_0})}=0.$$
Therefore, we obtain 
$$u^{(0)}=\int_0^{+\infty}u(t,\cdot)|_{B_{R_0}}dt=\lim_{t\to+\infty}v(t,\cdot)|_{B_{R_0}}\equiv0.$$
This proves that $u^{(0)}\equiv 0$ and we only need to show the remaining identities in  \eqref{l1b}.
Using the fact that $p\mapsto\hat{u}(p,\cdot)|_{B_{R_0}}\in \mathcal H(\mathbb C_{-\delta};H^1(B_{R_0}))$, we deduce that the map $$p\mapsto \mathcal A \hat{u}(p,\cdot)+c^{-2}p^2\hat{u}(p,\cdot)-pc^{-2}f$$ is lying in $\mathcal H(\mathbb D_\delta;H^{-1}(B_{R_0}))$ and  from \eqref{eq2} we find that
\begin{equation}\label{eq3}\mathcal A \hat{u}(p,x)+c^{-2}(x)p^2\hat{u}(p,x)= pc^{-2}(x)f(x),\quad x\in B_{R_0},\ p\in\mathbb D_\delta.\end{equation}
Inserting the expression \eqref{l1a} into \eqref{eq3} we obtain \eqref{l1b}.\end{proof}

Now let us consider, for $j=1,2$,  $f_j\in H^1(\R^3)$ and let $c_j\in L^\infty(\R^3)$ be  an admissible sound speed coefficient, such that supp$(f_j)\cup$  supp$(c_1-c_2)\subset\overline{\Omega}$ and $c_j$ is constant on  $\R^3\setminus\overline{\Omega}$. Using the above properties, we obtain the following.

\begin{lem}\label{l2} For $j=1,2$, let $f_j\in H^1(\R^3)$ be non-uniformly vanishing and let $c_j\in L^\infty(\R^3)$ be  an admissible sound speed coefficient, such that supp$(f_j)\cup$supp$(c_1-c_2)\subset\overline{\Omega}$ and $c_j$ is constant on  $\R^3\setminus\overline{\Omega}$. Let $u_j$ be the solution of \eqref{eq1} with $c=c_j$ and $f=f_j$.
Assuming that the condition
\begin{equation}\label{l2b}u_1(t,x)=u_2(t,x),\quad (t,x)\in\R_+\times\partial\Omega\end{equation}
 is fulfilled, for any $\phi\in H^1(\Omega)$ satisfying $\mathcal A \phi=0$ in $\Omega$, we obtain the following identities
\begin{equation}\label{l2c} \int_{\Omega}(c^{-2}_1f_1-c^{-2}_2f_2)\phi dx= \int_{\Omega}(c^{-2}_1u^{(k)}_1-c^{-2}_2u^{(k)}_2)\phi dx=0,\quad k\in\mathbb N.\end{equation}

\end{lem}
\begin{proof}
In view of Lemma \ref{l1}, fixing
\bel{l2aa}u_j^{(k)}(x)=\frac{(-1)^k\int_0^{+\infty}t^ku_j(t,x)dt}{k!},\quad x\in B_{R_0},\ k\in\mathbb N,\ j=1,2,\ee
we get
\begin{equation}\label{l2a}\hat{u_j}(p,\cdot)|_{B_{R_0}}=\sum_{k=1}^\infty u_j^{(k)}p^k,\quad p\in\mathbb D_\delta.\end{equation}
Now let us consider $u=u_1-u_2$ and notice that the restriction of $u$ to $\R_+\times (\R^3\setminus\overline{\Omega})$ solves the initial boundary value problem
$$\left\{\begin{array}{ll}c_1^{-2}(x)\partial_t^2u+\mathcal A u=0,\quad &\textrm{in}\ \R_+\times(\R^3\setminus\overline{\Omega}),\\  u(0,x)=0,\quad \partial_tu(0,x)=0,&x\in \R^3\setminus\overline{\Omega}\\
u(t,x)=0, &(t,x)\in \R_+\times\partial\Omega .\end{array}\right.$$
Then, the uniqueness of the solution of this initial boundary value problem implies that $u=0$ on $\R_+\times (\R^3\setminus\overline{\Omega})$ and it follows that $u_1=u_2$ on $\R_+\times (\R^3\setminus\overline{\Omega})$. Applying \eqref{l2aa}, we obtain
$$u_1^{(k)}(x)=\frac{(-1)^k\int_0^{+\infty}t^ku_1(t,x)dt}{k!}=\frac{(-1)^k\int_0^{+\infty}t^ku_2(t,x)dt}{k!}=u_2^{(k)}(x),\quad x\in B_{R_0}\setminus\overline{\Omega},\ k\in\mathbb N.$$  Therefore, applying \eqref{l1b} and fixing $u^{(k)}=u_1^{(k)}-u_2^{(k)}$, $k\in\mathbb N$, we get
$$\left\{\begin{array}{ll}\mathcal A u^{(1)}=c^{-2}_1f_1-c^{-2}_2f_2,\quad &\textrm{in}\ B_{R_0},\\  
u^{(1)}(x)=0, &x\in B_{R_0}\setminus\overline{\Omega} ,\end{array}\right.$$
and, for all $k\geq2$, 
\begin{equation}\label{l2e}\left\{\begin{array}{ll}\mathcal A u^{(k)}=c^{-2}_2u_2^{(k-2)}-c^{-2}_1u_1^{(k-2)},\quad &\textrm{in}\ B_{R_0},\\  
u^{(k)}(x)=0, &x\in B_{R_0}\setminus\overline{\Omega} .\end{array}\right.\end{equation}
Fixing $\phi\in H^1(\Omega)$, satisfying $\mathcal A \phi=0$ in $\Omega$, multiplying each of these equations by $\phi$  and integrating by parts, we obtain \eqref{l2c}.\end{proof}

\section{Proof of Corollary \ref{t1}}

This section is devoted to the proof of Corollary \ref{t1}. We start by considering the following intermediate result.

\begin{prop}\label{p1} Let  $f\in H^1(\R^3)$  and let $c\in L^\infty(\R^3)$ be  an admissible sound speed coefficient, such that  $c$ is constant on  $\R^3\setminus\overline{\Omega}$. Let $u$ be the solution of \eqref{eq1}. Then, the condition
\bel{p1a} \int_0^{+\infty}t^{2k}u(t,x)dt=0,\quad x\in B_{R_0},\ k\in\mathbb N\ee
implies that $f\equiv0$.\end{prop}
\begin{proof}
In view of Lemma \ref{l1}, condition \eqref{l1a} is fulfilled, we have $u^{(0)}\equiv0$ and
$$u^{(2k)}(x)=\frac{\partial_p^{2k}\hat{u}(p,x)|_{p=0}}{(2k)!}=\frac{\int_0^{+\infty}t^{2k}u(t,x)dt}{(2k)!},\quad x\in B_{R_0},\ k\in\mathbb N.$$
Then condition \eqref{p1a} implies that $u^{(2k)}\equiv0$, $k\in\mathbb N\cup\{0\}$ which implies that
$$\hat{u}(p,\cdot)|_{B_{R_0}}=\sum_{k=0}^\infty u^{(2k+1)}p^{2k+1},\quad p\in\mathbb D_\delta.$$
In particular, we obtain
$$\hat{u}(-p,\cdot)|_{B_{R_0}}=-\hat{u}(p,\cdot)|_{B_{R_0}},\quad p\in\mathbb D_\delta.$$
Using the fact that $p\mapsto\hat{u}(p,\cdot)|_{B_{R_0}}\in\mathcal H(\mathbb C_{-\delta};H^1(B_{R_0}))$, we deduce that
\bel{l3b}\hat{u}(-p,\cdot)|_{B_{R_0}}=-\hat{u}(p,\cdot)|_{B_{R_0}},\quad p\in\{z\in \mathbb C_{-\delta}:\ \re(z)\leq0\}.\ee
Now let us consider the map $v:\mathbb C\longrightarrow H^1(B_{R_0})$, defined by 
$$v(p,\cdot):=\left\{\begin{array}{l} \hat{u}(p,\cdot)|_{B_{R_0}}\ \textrm{if }p\in\mathbb C_+\\ -\hat{u}(-p,\cdot)|_{B_{R_0}}\ \textrm{if }p\in\mathbb C\setminus\mathbb C_+.\end{array}\right.$$
Combining \eqref{l3b} with the fact that $p\mapsto\hat{u}(p,\cdot)|_{B_{R_0}}\in\mathcal H(\mathbb C_{-\delta};H^1(B_{R_0}))$, we deduce that $p\mapsto v(p,\cdot)\in\mathcal H(\mathbb C;H^1(B_{R_0}))$ and it is the holomorphic extension of $\hat{u}(p,\cdot)|_{B_{R_0}}$ in the complex plan. Moreover, for all $p\in\overline{\mathbb C_+}$, we have the following estimate
\bel{l3g}\begin{aligned}\norm{v(p,\cdot)}_{H^1(B_{R_0})}+\norm{v(-p,\cdot)}_{H^1(B_{R_0})}&\leq 2\norm{\hat{u}(p,\cdot)}_{H^1(B_{R_0})}\\
&\leq C\int_0^{+\infty}e^{-\delta t}dt\norm{f}_{H^1(\R^3)}\leq C\norm{f}_{H^1(\R^3)},\end{aligned}\ee
with $C>0$ independent of $p$. Therefore, applying Liouville theorem we deduce that there exists $g\in H^1(B_{R_0})$ such that
$$v(p,x)=g(x),\quad x\in B_{R_0},\ p\in\mathbb C.$$
It follows that 
$$\hat{u}(p,x)|_{B_{R_0}}=g(x),\quad x\in B_{R_0},\ p\in\mathbb C_+$$
and the uniqueness of the Laplace transform implies that
$$u(t,x)=\delta_0(t)\otimes g(x),\quad x\in B_{R_0},\ t\in\mathbb R_+,$$
with $\delta_0$ the delta Dirac distribution. Combining this with the fact that $u\in C([0,+\infty);H^1(\R^3))$ we deduce that
$$u(t,x)=0,\quad x\in B_{R_0},\ t\in[0,+\infty).$$
Then, we find $f|_{B_{R_0}}=u(0,\cdot)|_{B_{R_0}}\equiv0$ and using the fact that supp$(f)\subset B_{R_0}$, we deduce that $f\equiv 0$. This  completes the proof of Proposition \ref{p1}.\end{proof}

Armed with Proposition \ref{p1} we are now in position to complete the proof of Corollary \ref{t1}.

\textbf{Proof of Corollary \ref{t1}.} We use the notation of Lemma \ref{l2} and without loss of generality we assume that $c_1\leq c_2$. We will prove that  \eqref{l2b} implies that $c_1=c_2$ and $f_1=f_2$. Note first that
$$u_1^{(2k)}(x)=\frac{\partial_p^{2k}\hat{u_1}(p,x)|_{p=0}}{(2k)!}=\frac{\int_0^{+\infty}t^{2k}u_1(t,x)dt}{(2k)!},\quad x\in B_{R_0}.$$
Then, since $f_1\not\equiv0$, we deduce from Proposition \ref{p1} that
\bel{E}E=\{k\in\mathbb N:\ u_1^{(2k)}|_{B_{R_0}}\not\equiv 0\}\neq\emptyset.\ee
We fix $k_0=\min E$ and we will complete the proof by considering the two different situation $k_0=1$ and $k_0\geq 2$.\\

\textbf{Case 1: $k_0=1$}. Following Lemma \ref{l1} and \ref{l2}, fixing $u^{(2)}=u_1^{(2)}-u_2^{(2)}$ we deduce that $u^{(2)}\in H^1(B_{R_0})$ and it satisfies $$\left\{\begin{array}{ll}\mathcal A u^{(2)}=0,\quad &\textrm{in}\ B_{R_0},\\  
u^{(2)}(x)=0, &x\in B_{R_0}\setminus\overline{\Omega} .\end{array}\right.$$ Therefore, applying results of unique continuation for elliptic equations we deduce that $u^{(2)}\equiv 0$ which implies that $u_1^{(2)}=u_2^{(2)}$ in $B_{R_0}$. Now applying \eqref{l2c} with $\phi=\overline{u_1^{(2)}}$ we obtain
$$0=\int_{\Omega}(c^{-2}_1u^{(2)}_1-c^{-2}_2u^{(2)}_2)\overline{u_1^{(2)}} dx=\int_{\Omega}(c^{-2}_1-c^{-2}_2)|u^{(2)}_1|^2 dx.$$
Combining this with the fact that $c_1\leq c_2$, we obtain $(c^{-2}_1-c^{-2}_2)|u^{(2)}_1|^2=0$ in $\Omega$ and, in view of \eqref{c}, we find
\begin{equation}\label{t1e}(c_2(x)-c_1(x))|u^{(2)}_1(x)|^2=0,\quad x\in\Omega.\end{equation}
We will show that this condition implies that $c_1=c_2$ in $\Omega$. Assuming the contrary and applying the fact that $c_1-c_2\in  C(\overline{\Omega})$, we deduce that the set 
$$\mathcal K:=\{x\in\Omega:\ c_2(x)-c_1(x)>0\}$$
is an open and not empty set of $\R^3$. Then, recalling that $\mathcal A u^{(2)}_1=0$  in $B_{R_0}$ with $\overline{\Omega}\subset B_{R_0}$,  we deduce that $u^{(2)}_1\in C(\overline{\Omega})$ and \eqref{t1e} implies that $u^{(2)}_1=0$ on $\mathcal K$. Thus the unique continuation principle implies that $u^{(2)}_1=0$ on $B_{R_0}$, contradicting the fact that $u_1^{(2)}|_{B_{R_0}}\not\equiv 0$. Therefore, we have $c_1=c_2$.

\textbf{Case 2: $k_0\geq2$}. In this case, we have 
$$u_1^{(0)}|_{B_{R_0}}=\ldots=u_1^{(2(k_0-1))}|_{B_{R_0}}\equiv 0.$$
In a similar way to case 1, we can prove that $u_2^{(0)}|_{B_{R_0}}=u_1^{(0)}|_{B_{R_0}}\equiv0$ and by iteration, we deduce that 
$u_2^{(0)}|_{B_{R_0}}=u_2^{(2)}|_{B_{R_0}}=\ldots=u_2^{(2(k_0-1))}|_{B_{R_0}}\equiv 0$. It follows that
\begin{equation}\label{t1f}u_j^{(0)}|_{B_{R_0}}=u_j^{(2)}|_{B_{R_0}}=\ldots=u_j^{(2(k_0-1))}|_{B_{R_0}}\equiv 0,\quad j=1,2.\end{equation}
 In addition, using \eqref{l2e} and fixing $u^{(2k_0)}=u_1^{(2k_0)}-u_2^{(2k_0)}$ we obtain
$$\left\{\begin{array}{ll}\mathcal A u^{(2k_0)}=0,\quad &\textrm{in}\ B_{R_0},\\  
u^{(2k_0)}(x)=0, &x\in B_{R_0}\setminus\overline{\Omega} .\end{array}\right.$$
Therefore, by unique continuation we deduce that $u_1^{(2k_0)}=u_2^{(2k_0)}$ in $B_{R_0}$. Applying \eqref{l2c} with $\phi=\overline{u_1^{(2k_0)}}$ we obtain
$$0=\int_{\Omega}(c^{-2}_1u^{(2k_0)}_1-c^{-2}_2u^{(2k_0)}_2)\overline{u_1^{(2k_0)}} dx=\int_{\Omega}(c^{-2}_1-c^{-2}_2)|u^{(2k_0)}_1|^2 dx.$$
Then, repeating the arguments of  case 1 we can show by contradiction that $c_1=c_2$.

In both cases we have $c_1=c_2$. In addition, applying  Theorem \ref{t2} of the appendix we deduce that \eqref{l2b} implies also that $f_1=f_2$ which proves that in all cases  \eqref{t1c} holds true. This completes the proof of Corollary \ref{t1}.

\section{Proof of Theorem \ref{t3}}

In this section we will use the notation of the proof of Corollary \ref{t1}. We will prove that  \eqref{l2b} implies that $c_1=c_2$ and $f_1=f_2$. In a similar way to Theorem \ref{t1}, we consider the set $E$ given by \eqref{E} and applying Proposition \ref{p1} we recall that $E$ is not empty. Then we fix $k_0=\min E$ and, in a similar way to Theorem \ref{t1}, one can check that  \eqref{t1f} is fulfilled and $u_1^{(2k_0)}=u_2^{(2k_0)}$ in $B_{R_0}$. Fixing $u^{(2k_0+2)}=u_1^{(2k_0+2)}-u_2^{(2k_0+2)}$ and applying \eqref{l2e} and \eqref{t3c}, we obtain 
$$\left\{\begin{array}{ll}\mathcal A u^{(2k_0+2)}(x)=(c_2^{-2}(x)-c_1^{-2}(x))u_1^{(2k_0)}(x)=0,\quad &x\in \mathcal O,\\  
u^{(2k_0+2)}(x)=0, &x\in B_{R_0}\setminus\overline{\Omega} .\end{array}\right.$$
On the other hand, since $K_j\subset \overline{\Omega}$, $j=1,\ldots,N$, we have $B_{R_0}\setminus\overline{\Omega}\subset \mathcal O$. Combining this with \eqref{t3c} and applying results of unique continuation we deduce that $u^{(2k_0+2)}=0$ on $\mathcal O$.

Now let us fix
$$r_1:=\underset{1\leq j\leq N}{\min} \textrm{dist}(K_j,\partial B_{R_0})>0,\quad r_2:=\underset{1\leq i<j\leq N}{\min} \textrm{dist}(K_i,K_j)>0,\quad r_0:=\min (r_1,r_2).$$
We define also the sets
\bel{t3aa}K_j':=\left\{x:\ \textrm{dist}(x,K_j)\leq \frac{r_0}{6}\right\},\quad U_j:=\left\{x:\ \textrm{dist}(x,K_j)< \frac{r_0}{3}\right\},\quad j=1,\ldots,N.\ee
Then, we consider $\chi_j\in C^\infty_0(U_j)$ satisfying $\chi_j=1$ on $K_j'$ and $0\leq\chi_j\leq 1$, $j=1,\ldots,N$. It is clear that
\bel{t3e} K_i'\cap \textrm{supp}(\chi_j)=\emptyset,\quad i,j=1,\ldots,N,\ i\neq j,\ee
\bel{t3f}  \textrm{supp}(\partial_{x_k}\chi_j)\subset\mathcal O ,\quad j=1,\ldots,N,\ k=1,2,3.\ee

In view of \eqref{t3e}, fixing $j=1,\ldots,N$ and multiplying \eqref{l2e}, with $k=2k_0+2$, by $\chi_j \overline{u_1^{(2k_0)}}$ and integrating by parts we obtain
$$\begin{aligned} &\int_{K_j}(c^{-2}_1-c^{-2}_2)|u^{(2k_0)}_1|^2 dx\\
&=\int_{B_{R_0}}\chi_j(c^{-2}_1-c^{-2}_2)|u^{(2k_0)}_1|^2 dx\\
&=\int_{B_{R_0}}(c^{-2}_1u^{(2k_0)}_1-c^{-2}_2u^{(2k_0)}_2)\chi_j\overline{u_1^{(2k_0)}} dx\\
&=-\int_{B_{R_0}}\mathcal A u^{(2k_0+2)}\chi_j\overline{u_1^{(2k_0)}} dx\\
&=-\int_{B_{R_0}}u^{(2k_0+2)}\left(2\sum_{m,n=1}^3a_{m,n}(x)\overline{\partial_{x_n}u^{(2k_0)}_1}\partial_{x_m}\chi_j+\chi_j\overline{\mathcal A u_1^{(2k_0)}}+\overline{u_1^{(2k_0)}}\mathcal A\chi_j \right)dx\\
&=-\int_{B_{R_0}}u^{(2k_0+2)}\left(2\sum_{m,n=1}^3a_{m,n}(x)\overline{\partial_{x_n}u^{(2k_0)}_1}\partial_{x_m}\chi_j+\overline{u_1^{(2k_0)}}\mathcal A\chi_j \right)dx.\end{aligned}$$
Moreover, using \eqref{t3f} and the fact that $u^{(2k_0+2)}=0$ on $\mathcal O$, we get
$$\int_{K_j}(c^{-2}_1-c^{-2}_2)|u^{(2k_0)}_1|^2 dx=-\int_{\mathcal O}u^{(2k_0+2)}\left(2\sum_{m,n=1}^3a_{mn}(x)\overline{\partial_{x_n}u^{(2k_0)}_1}\partial_{x_m}\chi_j+\overline{u_1^{(2k_0)}}\mathcal A\chi_j \right)dx=0.$$
Combining this with condition \eqref{t3d}, we deduce that
$$(c^{-2}_1(x)-c^{-2}_2(x))|u^{(2k_0)}_1(x)|^2=0,\quad x\in K_j$$
and repeating the arguments of Corollary \ref{t1} we deduce that $c_1=c_2$ on $K_j$. Therefore, we have
$$c_1(x)=c_2(x),\quad x\in \bigcup_{j=1}^N K_j$$
and \eqref{t3c} implies that $c_1=c_2$. Then, applying  Theorem \ref{t2} we find that \eqref{l2b} implies also that $f_1=f_2$ which proves   \eqref{t1c}. This completes the proof of Theorem \ref{t3}.

\section{Proof of Theorem \ref{t5},  Corollary \ref{c1} and Proposition \ref{r1}}
In this section we will use the notation of the proof of Corollary \ref{t1}. We start by considering the proof of Theorem \ref{t5} and Corollary \ref{c1}. Then we prove Proposition \ref{r1} by showing that the conditions of Theorem \ref{t5} are fulfilled when $\mathcal A=-\Delta$ and \eqref{kmm}, with $f=f_1$ and $c=c_1$, is fulfilled.  We prove the same result when \eqref{kmm} is not fulfilled but the function $h$ given by \eqref{h} is non-uniformly vanishing and of constant sign and we give a more general formulation of such conditions.

\textbf{Proof of Theorem \ref{t5}.} We will prove that  \eqref{l2b} implies that $c_1=c_2$ and $f_1=f_2$. Let us  first observe that since, for $j=1,2$, $f_j$ is real valued, we deduce that $u_j$ is real valued and by definition the functions $u_j^{(k)}$, $k\in\mathbb N$, are also real valued. In view of Proposition \ref{p1}, since $f_1\not\equiv0$, the set $\mathcal B$ defined by \eqref{t5a} is not empty and therefore we can fix $k_0=\min \mathcal B$. Then,  we complete the proof by considering the two different situation $k_0=1$ and $k_0\geq 2$.\\
\ \\
\textbf{Case 1:} $k_0=1$. Using the fact that
$$u_1^{(2)}(x)=\frac{\int_0^{+\infty}t^2u_1(t,x)dt}{2},\quad x\in B_{R_0},$$
we deduce that $u^{(2)}_1$ solves the boundary value problem
$$\left\{\begin{array}{ll}\mathcal A u^{(2)}_1=0,\quad &\textrm{in}\ \Omega,\\  
u^{(2)}_1(x)=\frac{g(x)}{2}, &x\in \partial\Omega ,\end{array}\right.$$
with $g$ given by \eqref{g}. Since $g$ is of constant sign the maximum principle implies that $u^{(2)}_1$ is also of constant sign and repeating the arguments used in the proof of Corollary \ref{t1} we can show that $u^{(2)}_1=u^{(2)}_2$. Combining this with \eqref{t5b} and Lemma \ref{l2}, we deduce that
$$\begin{aligned}0=\int_{\Omega}(c^{-2}_1u^{(2)}_1-c^{-2}_2u^{(2)}_2)\psi dx&=\int_{\Omega}(c^{-2}_1-c^{-2}_2)u^{(2)}_1\psi dx\\
\ &=\int_{\Omega}\left(\frac{(c_1+c_2)(c_2-c_1)}{c^{2}_1c^{2}_2}\right)u^{(2)}_1\psi dx\\
&=-\int_{\Omega}\left(\frac{(c_1+c_2)u^{(2)}_1}{c^{2}_1c^{2}_2}\right)\psi^2h dx.\end{aligned}$$
Combining this with the fact that $u^{(2)}_1$ and $h$ are of constant sign, we deduce that
$$\left(\frac{(c_1+c_2)u^{(2)}_1}{c^{2}_1c^{2}_2}\right)\psi^2h(x)=0,\quad x\in\Omega$$
which implies that
\bel{t5d}u^{(2)}_1(x)\psi^2(x)h(x)=0,\quad x\in\Omega.\ee
From now on we will use this identity for proving that $c_1=c_2$. It is clear that if $h=0$ on $\Omega$ we have $c_1=c_2$. So let us assume the contrary. Then we can define the open and non-empty set $\mathcal U:=\{x\in\Omega:\ |h(x)|>0\}$. In view of \eqref{t5d}, we have
$$u^{(2)}_1(x)\psi^2(x)=0,\quad x\in\mathcal U.$$
Recalling that $\mathcal A u^{(2)}_1=0$ on $B_{R_0}$ with $\overline{\Omega}\subset B_{R_0}$, we deduce that $u^{(2)}_1\in C^1(\overline{\Omega})$ and since $g\not\equiv0$, we deduce that $u^{(2)}_1|_{\Omega}\not\equiv0$. Applying the strong maximum principle we deduce that $|u^{(2)}_1(x)|>0$, $x\in\Omega$ and it follows that $\psi(x)=0$, $x\in\mathcal U$. Using the fact that $\Omega$ is connected and applying the unique continuation for elliptic equations, we deduce that $\psi=0$ on $\Omega$ which implies that $c_1=c_2$.\\
\ \\
\textbf{Case 2:} $k_0\geq 2$. In that case, we have $u^{(2)}_1\in H^1_0(\Omega)$ and $\mathcal A u^{(2)}_1=0$ on $\Omega$. This implies that $u^{(2)}_1=0$ in $\Omega$ and then by unique continuation on $B_{R_0}$. Repeating the arguments of the first step we deduce that $u^{(2)}_2=u^{(2)}_1\equiv0$. In the same way, applying Lemma \ref{l1} we can prove by iteration that \eqref{t1f} holds true. Combining this with \eqref{l2b} and Lemma \ref{l1}, we deduce that $u^{(2k_0)}_2=u^{(2k_0)}_1$ and $u^{(2k_0)}_1$ solves the boundary value problem
$$\left\{\begin{array}{ll}\mathcal A u^{(2k_0)}_1=0,\quad &\textrm{in}\ \Omega,\\  
u^{(2k_0)}_1(x)=\frac{g(x)}{(2k_0)!}, &x\in \partial\Omega .\end{array}\right.$$
Therefore, applying \eqref{t5b} and Lemma \ref{l2}, we obtain
$$\begin{aligned}0=\int_{\Omega}(c^{-2}_1u^{(2k_0)}_1-c^{-2}_2u^{(2k_0)}_2)\psi dx&=\int_{\Omega}(c^{-2}_1-c^{-2}_2)u^{(2k_0)}_1\psi dx\\
&=-\int_{\Omega}\left(\frac{(c_1+c_2)u^{(2k_0)}_1}{c^{2}_1c^{2}_2}\right)\psi^2h dx.\end{aligned}$$
Thus, repeating the arguments of the first case, we get $c_1=c_2$.

In both cases we have $c_1=c_2$. In addition, applying  Theorem \ref{t2} we deduce that \eqref{l2b} implies also that $f_1=f_2$ which proves that in all case  \eqref{t1c} holds true. This completes the proof of Theorem \ref{t5}.\qed
\ \\

\textbf{Proof of Corollary \ref{c1}.} Consider the set $K_j'$ and $U_j$, $j=1,\ldots,N$, given by \eqref{t3aa} and consider $\chi_j\in C^\infty_0(U_j)$ satisfying $\chi_j=1$ on $K_j'$ and $0\leq\chi_j\leq 1$, $j=1,\ldots,N$. In a similar way to Theorem \ref{t3}, applying \eqref{t3e}-\eqref{t3f}, multiplying \eqref{l2e}, with $k=2k_0+2$, by $\chi_j \psi_j$ and integrating by parts we obtain
$$\int_{K_j}(c^{-2}_1-c^{-2}_2)u^{(2k_0)}_1 \psi_j dx=-\int_{\mathcal O}u^{(2k_0+2)}\left(2\sum_{m,n=1}^3a_{mn}(x)\partial_{x_n}\psi_j\partial_{x_m}\chi_j+\psi_j\mathcal A\chi_j \right)dx=0.$$
Applying \eqref{c1a}, we get
$$\int_{K_j}\left(\frac{(c_1+c_2)u^{(2k_0)}_1}{c^{2}_1c^{2}_2}\right)\psi_j^2h_j dx=0,\quad j=1,\ldots,N.$$
Then repeating the arguments used in the proof of Theorem \ref{t5} we can prove that $\psi_jh_j\equiv0$, $j=1,\ldots,N$. Combining this with \eqref{t3f} and \eqref{c1a}, we  deduce that $c_1=c_2$ and applying Theorem \ref{t2} we deduce that  \eqref{t1c} holds true. This completes the proof of Corollary \ref{c1}.\qed

\ \\
\textbf{Proof of Proposition \ref{r1}.} We use the notation of Corollary \ref{t1}. We will prove Proposition \ref{r1},  in the case $\mathcal A=-\Delta$, by following the analysis of \cite{LU} based on arguments of scattering theory and properties of representation of solutions of Helmholtz equations. For the proof of this proposition, we extend the approach of \cite{KM,LU} to the case of sound speed coefficients which are not equal to $1$ on $\R^3\setminus\overline{\Omega}$ and we give a proof that will combine the properties of Lemma \ref{l1} with some arguments of \cite{KM,LU}. Moreover, we extend the asymptotic extension in low frequency of Fourier transform in time of the solution of \eqref{eq1} considered by \cite{KM,LU} and we prove that when \eqref{kmm} is not fulfilled the sign condition imposed to \eqref{h} implies the condition of sign imposed to \eqref{g} in Theorem \ref{t5} and Corollary \ref{c1}. 

Let us first consider the constant $c_0>0$ such that 
$$c_1(x)=c_0,\quad x\in\R^3\setminus\overline{\Omega}.$$
We fix $v$ defined by
$$v(t,x)=u_1(c_0^{-1}t,x),\quad t\in\R_+,\ x\in\R^3$$
and one can check that $v$ solves \eqref{eq1} with $c=\frac{c_1}{c_0}$ and $f=f_1$. Then, using the fact that $\frac{c_1}{c_0}=1$ on  $\R^3\setminus\overline{\Omega}$ and applying \cite[Lemma 3.1]{LU}, for all $x\in\R^3$ and all $\tau\in\R$, we obtain 
$$\hat{v}(i\tau ,x)=- \tau^2\int_{B_{R_0}}\left(1-\frac{c_0^2}{c_1(y)^2}\right)\hat{v}(i\tau ,y)\Phi_\tau(x-y)dy-i \tau\int_{B_{R_0}}\frac{c_0^2f_1(y)}{c_1(y)^2}\Phi_\tau(x-y)dy,$$
where 
$$\Phi_\tau(x)=\frac{e^{i\tau |x|}}{4\pi |x|},\quad x\in\R^3\setminus\{0\},\ \tau\in\R.$$
On the other hand, we have $\hat{v}(i\tau ,\cdot)=c_0\hat{u_1}(ic_0\tau,\cdot)$ and, fixing $\tau_1=c_0\tau$, it follows that
\bel{r1a}\hat{u_1}(i\tau _1,x)=- \tau_1^2\int_{B_{R_0}}\left(c_0^{-2}-c_1(y)^{-2}\right)\hat{u_1}(i\tau _1,y)\Phi_{c_0^{-1}\tau_1}(x-y)dy-i\tau _1\int_{B_{R_0}}\frac{f_1(y)}{c_1(y)^2}\Phi_{c_0^{-1}\tau_1}(x-y)dy.\ee
Applying Taylor's formula, we  find
\bel{r1b}\abs{\Phi_{c_0^{-1}\tau_1}(x)-\frac{1}{4\pi |x|}-\frac{ic_0^{-1}\tau_1}{4\pi}}\leq \tau_1^2c_0^2|x|\leq 2\tau_1^2c_0^2R_0,\quad \tau_1\in\R,\ x\in B_{2R_0}\setminus\{0\}.\ee
Moreover, applying Lemma \ref{l1}, we deduce that
\bel{r1c}\norm{\hat{u_1}(i\tau _1,\cdot)|_{B_{R_0}} -i\tau _1u_1^{(1)}+\tau_1^2u_1^{(2)}}_{L^2(B_{R_0})}\leq C|\tau_1|^3,\quad \tau_1\in(-\delta/2,\delta/2),\ee
with $C>0$ independent of $\tau_1$ and $\delta>0$ the constant of estimate \eqref{d1a} with $R=R_0$. In addition, for all $x\in B_{R_0}$, we have 
$$\begin{aligned}\abs{\int_{B_{R_0}}\left(1-\frac{c_0^2}{c_1(y)^2}\right)\hat{u_1}(i\tau _1,y)\Phi_{c_0^{-1}\tau_1}(x-y)dy}&\leq C\int_{B_{R_0}}\abs{\hat{u_1}(i\tau _1,y)}\abs{x-y}^{-1}\chi_{B_{2R_0}}(x-y)dy\\
&\leq C(\abs{\hat{u_1}(i\tau _1,x)}\chi_{B_{R_0}}(x))*(\abs{x}^{-1}\chi_{B_{2R_0}}(x)),\end{aligned}$$
with $\chi_{A}$ the characteristic function of the set $A$ and $*$ the convolution product. Therefore, applying Young inequality and \eqref{r1c}, we obtain
$$\begin{aligned}\norm{\int_{B_{R_0}}\left(1-\frac{c_0^2}{c_1(y)^2}\right)\hat{u_1}(i\tau _1,y)\Phi_{c_0^{-1}\tau_1}(\cdot-y)dy}_{L^2(B_{R_0})}&\leq C\left(\int_{B_{2R_0}}\abs{x}^{-1}dx\right)\norm{\hat{u_1}(i\tau _1,\cdot)}_{L^2(B_{R_0})}\\
&\leq C\tau_1,\quad \tau_1\in(-\delta/2,\delta/2),\end{aligned}$$
where $C>0$ is independent of $\tau_1$. Combining this estimate with \eqref{r1a}-\eqref{r1b}, for all $\tau_1\in(-\delta/2,\delta/2)$, we obtain
$$\begin{aligned}&\norm{\hat{u_1}(i\tau _1,\cdot)|_{B_{R_0}} +i\tau _1\int_{B_{R_0}}\frac{f_1(y)}{4\pi c_1(y)^2|\cdot-y|}dy-\tau_1^2\int_{\R^3}\frac{f_1(y)}{4\pi c_0c_1(y)^2}dy}_{L^2(B_{R_0})}\\
&\leq C\tau_1^3,\end{aligned}$$
with $C>0$ independent of $\tau_1$. Then, from \eqref{r1c}, we obtain
\bel{r1e}u_1^{(1)}(x)=-\int_{B_{R_0}}\frac{c_1(y)^{-2}f_1(y)}{4\pi |x-y|}dy,\quad u_1^{(2)}(x)=-\frac{1}{4\pi c_0}\int_{\R^3}\frac{f_1(y)}{c_1(y)^2}dy,\quad x\in B_{R_0}.\ee
Thus, \eqref{kmm} implies that the set $\mathcal B$ given by \eqref{t5a} is not empty, $\min\mathcal B=1$ and $g$, given by \eqref{g} with $k_0=1$, is the following non-vanishing constant function  
$$g(x)=2u_1^{(2)}(x)=-\frac{1}{2\pi c_0}\int_{\R^3}\frac{f_1(y)}{c_1(y)^2}dy,\quad x\in \partial\Omega.$$
This proves that \eqref{kmm} implies the condition on the sign of \eqref{g} appearing in Theorem \ref{t5} and Corollary \ref{c1}. 

Now let us assume that condition \eqref{kmm} is not fulfilled. Then, in view of \eqref{r1e}, we have $u_1^{(2)}\equiv 0$ and
$$\norm{\hat{u_1}(i\tau _1,\cdot)|_{B_{R_0}} -i\tau _1u_1^{(1)}+i\tau _1^3u_1^{(3)}-u_1^{(4)}\tau_1^4}_{L^2(B_{R_0})}\leq C|\tau_1|^5,\quad \tau_1\in(-\delta/2,\delta/2).$$
Repeating the above argumentation, we find
$$u_1^{(4)}(x)=-\frac{1}{4c_0^{3}\pi}\int_{\R^3}|x-y|^2\frac{f_1(y)}{c_1(y)^2}dy+\frac{1}{4c_0\pi}\int_{\R^3}\left(c_0^{-2}-c_1(y)^{-2}\right)u_1^{(1)}(y)dy,\quad x\in B_{R_0},$$
which, combined with \eqref{r1e}, implies that
$$u_1^{(4)}(x)=-\frac{h(x)}{4c_0^{3}\pi},\quad x\in\partial\Omega,$$
with $h$ given by \eqref{h}. Therefore, assuming that $h$ is non-uniformly vanishing we deduce that $\mathcal B\neq\emptyset$, $\min\mathcal B=2$ and $g$, given by \eqref{g} with $k_0=2$, is defined by
$$g(x)=-\frac{6h(x)}{c_0^{3}\pi},\quad x\in\partial\Omega.$$
Then, if $h$ is of constant sign, $g$ will be also of constant sign and we will deduce that this condition implies the condition on the sign of \eqref{g} appearing in Theorem \ref{t5} and Corollary \ref{c1}. 

In the same way, following the proof of Theorem \ref{t5}, we know that $\mathcal B\neq\emptyset$. Therefore, fixing $k_0=\min\mathcal B$ and with $k_0\geq3$, we can prove  by iteration that
$$\begin{aligned}u_1^{(2k_0)}(x)=&-\frac{1}{4c_0^{2k_0-1}\pi}\int_{\R^3}|x-y|^{2(k_0-1)}\frac{f_1(y)}{c_1(y)^2}dy\\
&+\sum_{k=0}^{k_0-2}\frac{1}{4c_0^{2(k_0-k-1)-1}\pi}\int_{\R^3}\left(c_0^{-2}-c_1(y)^{-2}\right)|x-y|^{2(k_0-k-2)}u_1^{(2k+1)}(y)dy,\quad x\in B_{R_0}\end{aligned}$$
and it follows that $g$ given by \eqref{g} takes the form 
\bel{k0}\begin{aligned}g(x)=&-\frac{(2k_0)!}{4c_0^{2k_0-1}\pi}\int_{\R^3}|x-y|^{2(k_0-1)}\frac{f_1(y)}{c_1(y)^2}dy\\
&+\sum_{k=0}^{k_0-2}\frac{(2k_0)!}{4c_0^{2(k_0-k-1)-1}\pi}\int_{\R^3}\left(c_0^{-2}-c_1(y)^{-2}\right)|x-y|^{2(k_0-k-2)}u_1^{(2k+1)}(y)dy,\quad x\in \partial\Omega.\end{aligned}\ee
Therefore, by imposing a sign condition to this function we will be in position to apply the results of Theorem \ref{t5} and Corollary \ref{c1}. Note that the expression \eqref{k0} depends on $c_1$, $f_1$, $u_1^{(2k+1)}$, $k=0,\ldots,k_0-2$. Therefore, repeating the above arguments, we can derive by iteration the explicit expression of \eqref{k0}  in terms of an expression depending only on $c_1$ and $f_1$.

 In light of Theorem \ref{t5} and Corollary \ref{c1}, the condition \eqref{kmm} or the condition on the sign of \eqref{h} when \eqref{kmm} is not fulfilled imply that \eqref{t1c} holds true provided that condition \eqref{t5b} or \eqref{c1a} is fulfilled. According to the above argumentation, one can also extends this approach to the general case $\min\mathcal B=k_0$ for some general integer $k_0\in\mathbb N$.\qed

\section{Connection with the transmission eigenvalues problem}
In this section, we establish a connection between the TAT and PAT problem and the transmission eigenvalues problem investigated by many authors (see e.g. \cite{CCH,CGH,NN,PV, R1,Vo3}). A first connection  between these   problems has already been considered by \cite{FH} and we will give an extension of the result of \cite{FH}.  For this purpose, let us recall the definition of an eigenvalue of the transmission eigenvalues problem associated with \eqref{eq1}.

\begin{Def}\label{d2} A complex number $\tau\in\mathbb C$ is  called an eigenvalue of the transmission eigenvalues problem associated
with $c_1,c_2\in L^\infty (\R^3)$ on $\Omega$ if there is a non-zero pair of functions $(w_1,w_2)\in H^1(B_{R_0})^2$ satisfying the conditions
\bel{d2a}\left\{\begin{array}{ll}\mathcal A w_1-\tau^2c^{-2}_1w_1=0,\quad &\textrm{in}\ \Omega,\\
\mathcal A w_2-\tau^2c^{-2}_2w_2=0,\quad &\textrm{in}\ \Omega,\\  
w_1(x)=w_2(x),\ \partial_{\nu_a}w_1(x)=\partial_{\nu_a}w_2(x), &x\in\partial  \Omega,\end{array}\right.\ee
where $\partial_{\nu_a}$ denotes the conormal derivative associated with the coefficient $a$ of $\partial\Omega$ defined by
$$\partial_{\nu_a}v(x)=\sum_{j=1}^3a_{i,j}(x)\partial_{x_j}v(x)\nu_i(x),\quad x\in\partial\Omega.$$
\end{Def}
Fixing $r>0$ and $c_1,c_2\in L^\infty (\R^3)$, we introduce the set $NTE(c_1,c_2,r)$ of  complex numbers $\tau\in\mathbb C$ such that $\im(\tau)\in(-r,r)$ and $\tau$ is not an eigenvalue of the transmission eigenvalues problem associated with $c_1$ and $c_2$ on $\Omega$. We  establish the following link between our problem and the transmission eigenvalues problem \eqref{d2a}.

\begin{Thm}\label{t4} 
For $j=1,2$, let $f_j\in H^1(\R^3)$ be non-uniformly vanishing and let $c_j\in L^\infty(\R^3)$ be  an admissible sound speed coefficient, such that supp$(f_j)\cup$supp$(c_1-c_2)\subset\overline{\Omega}$ and $c_j$ is constant on  $\R^3\setminus\overline{\Omega}$. Let $u_j$ be the solution of \eqref{eq1} with $f=f_j$ and assume that
\begin{equation}\label{t4b}  u_1(t,x)=u_2(t,x),\quad t\in\R_+,\ x\in\partial\Omega.\end{equation}
We consider also $\delta>0$ such that  \eqref{d1a} holds true, with $R=R_0$, for all solutions of \eqref{tata} with $c=c_j$, $j=1,2$.
Then,  the set $NTE(c_1,c_2,\delta)$ consists of a set of at most discrete and isolated complex numbers $\tau\in\mathbb C$ such that $\im(\tau)\in(-\delta,\delta)$.
\end{Thm}
\begin{proof} For $r>0$, we fix $P_r:=\{z\in\mathbb C:\  \im(z)\in(-r,r)\}$. Following Lemma \ref{l1}, we define 
\bel{t4c}w_j(\tau,\cdot):=\hat{u_j}(i\tau ,\cdot)|_{B_{R_0}}+\hat{u_j}(-i\tau ,\cdot)|_{B_{R_0}},\quad \tau\in P_\delta.\ee
Then, for all $\tau\in P_\delta$, we have $(w_1(\tau,\cdot),w_2(\tau,\cdot))\in H^1(B_{R_0})^2$ and
\eqref{eq2} implies that
$$\mathcal A w_j(\tau,x)-c_j^{-2}(x)\tau^2w_j(\tau,x)=i\tau c_j^{-2}(x)f_j(x)-i\tau c_j^{-2}(x)f_j(x)=0,\quad x\in B_{R_0}.$$
Moreover, in a similar way to Lemma \ref{l2}, we can show that \eqref{t4b} implies
$$w_1(\tau,x)=w_2(\tau,x),\quad x\in B_{R_0}\setminus\overline{\Omega},\ \tau\in P_\delta.$$
Therefore, for all $\tau\in P_\delta$,  $(w_1(\tau,\cdot),w_2(\tau,\cdot))\in H^1(B_{R_0})^2$ solves \eqref{d2a} and in view of Lemma \ref{l1} 
we have $\tau\mapsto(w_1(\tau,\cdot),w_2(\tau,\cdot))\in\mathcal H(P_\delta,H^1(B_{R_0}))$. In addition, applying \eqref{d1a}, \eqref{l1c} and fixing
$$ u_j^{(2n)}=\frac{\int_0^{+\infty}t^{2n}u_j(t,\cdot)|_{B_{R_0}}dt}{(2n)!},\quad n\in\mathbb N\cup\{0\},$$
we deduce that the serie
$$\sum_{n=0}^N u_j^{(2n)}\tau^{2n},\quad N\in\mathbb N,\ j=1,2,$$
converges uniformly with respect to $\tau\in \overline{\mathbb D_{\delta_1}}$, $\delta_1\in(0,\delta)$,  as a map taking values in $H^1(B_{R_0})$ and Lemma \ref{l1} implies that
$$-2\sum_{n=0}^\infty u_j^{(2n)}\tau^{2n}=\hat{u_j}(i\tau ,\cdot)|_{B_{R_0}}+\hat{u_j}(-i\tau ,\cdot)|_{B_{R_0}}=w_j(\tau,\cdot),\quad \tau\in \mathbb D_\delta,\ j=1,2.$$
Then, in view of Proposition \ref{p1}, since $f_j$, $j=1,2$, is non-uniformly vanishing, the set of zero of the map $\tau\mapsto w_j(\tau,\cdot)$ in $\mathbb D_\delta$ are isolated. Thus,  the set of zero of the map $\tau\mapsto w_j(\tau,\cdot)$, $j=1,2$, in $P_{\delta}$ should also be isolated. Moreover,  one can easily check that for all $\tau\in P_{\delta}$ we have
$$w_1(\tau,\cdot)\equiv 0\Longleftrightarrow w_2(\tau,\cdot)\equiv 0.$$
Therefore,  the set of zero of the map $\tau\mapsto (w_1(\tau,\cdot),w_2(\tau,\cdot))$  in $P_{\delta}$ are isolated which implies that $NTE(c_1,c_2,\delta)$ is at most a discrete set with isolated complex numbers.\end{proof}

\begin{rem} In view of Theorem \ref{t4}, the condition \eqref{t4b} implies that the there exists $\delta>0$ such that the set of complex numbers $\tau\in\mathbb C$ with $\im(\tau)\in(-\delta,\delta)$ which are not an eigenvalue of the transmission eigenvalues problem associated
with $c_1,c_2\in L^\infty (\R^3)$ on $\Omega$ are at most discrete and isolated. The next step in the analysis of the TAT and PAT problem will be to extract information about $c_1-c_2$ from such properties of the corresponding transmission eigenvalues.\end{rem}

\section{Appendix}
In this appendix we consider the determination of the initial state when the sound speed coefficient $c$ is known. More precisely, let us consider the IVP
\begin{equation}\label{eq6}\left\{\begin{array}{ll}c^{-2}(x)\partial_t^2u+\mathcal A u=0,\quad &\textrm{in}\ \R_+\times\R^3,\\  u(0,x)=f(x),\quad \partial_tu(0,x)=g(x),&x\in\R^3\end{array}\right.\end{equation}
with $f\in H^1(\R^3)$ and $g\in L^2(\R^3)$ compactly supported and $c\in L^\infty(\R^3)$ taking a fix  constant value outside $\Omega$. Assuming that $c$ is known we determine $(f,g)$. This result can be stated as follows.

\begin{Thm}\label{t2} 
For $j=1,2$, let $f_j\in H^1(\R^3)$, $g_j\in L^2(\R^3)$ be  such that supp$(f_j)\cup$supp$(g_j)\subset\Omega$ and let $c\in L^\infty(\R^3)$ be  constant valued on $\R^3\setminus\overline{\Omega}$ and satisfy \eqref{c}. Let $u_j$ be the solution of \eqref{eq6} with  $f=f_j$, $g=g_j$. Then, the following implication holds true
\begin{equation}\label{t2c}  (u_1= u_2 \textrm{ on }\R_+\times\partial\Omega)\Rightarrow (f_1=f_2 \textrm{ and }g_1=g_2).\end{equation}

\end{Thm}

Let us observe that when $c$ is lying in $C^1(\R^3)$, one can prove Theorem \ref{t2} by combining the unique continuation of \cite{Ta} with suitable properties of hyperbolic equations (see e.g. \cite{ET,HK}). Nevertheless, it is not clear that such unique continuation results holds true when $c\in L^\infty(\R^3)$ and we are not aware of any result proving \eqref{t2c} for bounded sound speed coefficient $c$. For this reason we give the full proof of Theorem \ref{t2}

\textbf{Proof of Theorem \ref{t2}.} We assume that the condition \eqref{l2b} is fulfilled and we will prove that $f_1=f_2$ and $g_1=g_2$. Let us first consider $u=u_1-u_2$ and notice that $u$ solves \eqref{eq6} with $f=f_1-f_2$ and $g=g_1-g_2$. Combining \eqref{l2b} with the fact that supp$(f_j)\cup$supp$(g_j)\subset\Omega$,  we obtain in a similar way to Lemma \ref{l2} that
\begin{equation}\label{t2d}u(t,x)=0,\quad (t,x)\in\R_+\times (\R^3\setminus\overline{\Omega}).\end{equation}
Combining this with \eqref{t2d}, we deduce that, for all $p\in\mathbb C_+$, $\hat{u}(p,\cdot)\in H^1(\R^3)$ satisfies the following 
conditions 
\begin{equation}\label{t2e}\left\{\begin{array}{ll}-c^2(x)\mathcal A \hat{u}(p,x)+p^2\hat{u}(p,x)=g(x)+pf(x),\quad &\textrm{in}\ x\in\R^3,\\  
\hat{u}(p,x)=0, &x\in \R^3\setminus\overline{\Omega} .\end{array}\right.\end{equation}
Let us consider the operator $A$ acting on $L^2(B_{R_0};c^{-2}dx)$ with domain $D(A)=H^2(B_{R_0})\cap H^1_0(B_{R_0})$ defined  by
\[
A w:=c^2\mathcal A w,\quad w\in D(A).
\]
Recall that here we associate with the weighted space $L^2(B_{R_0};c^{-2}dx)$ the scalar product
$$\left\langle f,g\right\rangle_{L^2(B_{R_0};c^{-2}dx)}=\int_{B_{R_0}}f\overline{g}c^{-2}dx.$$
It is well known that the $A$ is a selfadjoint operator with a compact resolvent. We fix $\{\lambda_k\}_{k\in\mathbb N}$ and $m_k\in\mathbb N$  the strictly increasing and positive sequence of the eigenvalues of $A$ and the algebraic multiplicity of $\lambda_k$, respectively. For each eigenvalue $\lambda_k$, we introduce a family $\{\phi_{k,\ell}\}_{\ell=1}^{m_k}$ of eigenfunctions of $A$
which forms an orthonormal basis in $L^2(B_{R_0};c^{-2}dx)$ of the algebraic eigenspace of $A$ associated with $\lambda_k$. Multiplying \eqref{t2e} by $c^{-2}\phi_{k,\ell}$, $k\in\mathbb N$, $\ell=1,\ldots,m_k$, and integrating by parts on $B_{R_0}$, we obtain
$$(\lambda_k+p^2)\left\langle \hat{u}(p,\cdot),\phi_{k,\ell}\right\rangle_{L^2(B_{R_0};c^{-2}dx)}=\left\langle g,\phi_{k,\ell}\right\rangle_{L^2(B_{R_0};c^{-2}dx)}+ p\left\langle f,\phi_{k,\ell}\right\rangle_{L^2(B_{R_0};c^{-2}dx)}.$$
It follows that, for all $p\in \mathbb C_+$, we have
$$\hat{u}(p,\cdot)|_{B_{R_0}}=\sum_{k=1}^\infty\frac{\sum_{\ell=1}^{m_k}\left(\left\langle g,\phi_{k,\ell}\right\rangle_{L^2(B_{R_0};c^{-2}dx)}+ p\left\langle f,\phi_{k,\ell}\right\rangle_{L^2(B_{R_0};c^{-2}dx)}\right)\phi_{k,\ell}}{\lambda_k+p^2}.$$
Fix $\mathcal U=\mathbb C\setminus\{\pm i\sqrt{\lambda_k}\}$. It is clear that the map $p\mapsto \hat{u}(p,\cdot)|_{B_{R_0}}$ can be extended analytically to an element of $\mathcal H(\mathcal U;L^2(B_{R_0};c^{-2}dx))$. Combining this with \eqref{t2e} and applying the unique continuation property for analytic functions, we deduce that 
\begin{equation}\label{t2ee}\sum_{k=1}^\infty\frac{\sum_{\ell=1}^{m_k}\left(\left\langle g,\phi_{k,\ell}\right\rangle_{L^2(B_{R_0};c^{-2}dx)}+ p\left\langle f,\phi_{k,\ell}\right\rangle_{L^2(B_{R_0};c^{-2}dx)}\right)\phi_{k,\ell}}{\lambda_k+p^2}|_{B_{R_0}\setminus\overline{\Omega}}=0,\quad p\in \mathcal U.\end{equation} Multiplying this expression by $p+i\sqrt{\lambda_k}$ and sending $p\to-i\sqrt{\lambda_k}$, we obtain
$$\frac{\sum_{\ell=1}^{m_k}\left(\left\langle g,\phi_{k,\ell}\right\rangle_{L^2(B_{R_0};c^{-2}dx)}-i\sqrt{\lambda_k}\left\langle f,\phi_{k,\ell}\right\rangle_{L^2(B_{R_0};c^{-2}dx)}\right)\phi_{k,\ell}|_{B_{R_0}\setminus\overline{\Omega}}}{-2i\sqrt{\lambda_k}}=0,\quad k\in\mathbb N.$$
Moreover, as a simple consequence of unique continuation property for elliptic equations, we know that the functions $\phi_{k,\ell}|_{B_{R_0}\setminus\overline{\Omega}}$, $\ell=1,\ldots,m_k$, are linearly independent as element of $L^2(B_{R_0}\setminus\overline{\Omega})$ (see e.g. \cite[Step 4 in the proof of Theorem 1.1]{KSXY}). It follows that 
\begin{equation}\label{t2f}\left\langle g,\phi_{k,\ell}\right\rangle_{L^2(B_{R_0};c^{-2}dx)}=i\sqrt{\lambda_k}\left\langle f,\phi_{k,\ell}\right\rangle_{L^2(B_{R_0};c^{-2}dx)},\quad k\in\mathbb N,\ \ell=1,\ldots,m_k.\end{equation}
In the same way multiplying \eqref{t2e} by $p-i\sqrt{\lambda_k}$ and sending $p\to i\sqrt{\lambda_k}$, we obtain
$$\frac{\sum_{\ell=1}^{m_k}\left(\left\langle g,\phi_{k,\ell}\right\rangle_{L^2(B_{R_0};c^{-2}dx)}+i\sqrt{\lambda_k}\left\langle f,\phi_{k,\ell}\right\rangle_{L^2(B_{R_0};c^{-2}dx)}\right)\phi_{k,\ell}|_{B_{R_0}\setminus\overline{\Omega}}}{2i\sqrt{\lambda_k}}=0,$$
which implies that
$$\left\langle g,\phi_{k,\ell}\right\rangle_{L^2(B_{R_0};c^{-2}dx)}=-i\sqrt{\lambda_k}\left\langle f,\phi_{k,\ell}\right\rangle_{L^2(B_{R_0};c^{-2}dx)},\quad k\in\mathbb N,\ \ell=1,\ldots,m_k.$$
Combining this with \eqref{t2f} we obtain
$$\left\langle g,\phi_{k,\ell}\right\rangle_{L^2(B_{R_0};c^{-2}dx)}=0,\quad k\in\mathbb N,\ \ell=1,\ldots,m_k$$
and we obtain  $g\equiv0$. Then, from \eqref{t2f} we deduce that $f\equiv0$. This proves that \eqref{t2c} holds true and it completes the proof of Theorem \ref{t2}.\qed

\section*{Acknowledgements}

 The research of G.U. is partly supported by NSF and a Robert R, Phelps and Elaine F. Phelps Professorhip at University of Washington.


\begin{thebibliography}{99}
%

\bibitem{AKK} {\sc M. Agranovsky, P. Kuchment and L. Kunyansky},  {\em On reconstruction formulas and algorithms
for the thermoacoustic tomography Photoacoustic Imaging and Spectroscopy}, Boca Raton, FL:CRC Press, 2009, pp 89–101.
 \bibitem{CCH} {\sc F. Cakoni, D. Colton, H. Haddar}, {\em Inverse Scattering Theory and Transmission
Eigenvalues}, CBMS-NSF Regional Conference Series in Applied Mathematics, vol. 88, Society for
Industrial and Applied Mathematics (SIAM), Philadelphia, PA, 2016. 
 \bibitem{CGH} {\sc  F. Cakoni, D. Gintides, H. Haddar}, {\em The existence of an infinite discrete set of
transmission eigenvalues}, SIAM J. Math. Anal., \textbf{42} (1) (2010) 237-255. 

\bibitem{DSK} {\sc G. J. Diebold , T. Sun and M. I. Khan}, {\em Photoacoustic monopole radiation in one, two, and three dimensions}, 
Phys. Rev. Lett., \textbf{67} (1991), 3384-3387.
\bibitem{ET} {\sc M. Eller and D. Toundykov}, {\em A global Holmgren theorem for multidimensional hyperbolic partial differential equations}, Applicable Analysis, \textbf{91} (2012), 69-90.
\bibitem{FH}{\sc D. Finch and K. Hickmann}, {\em Transmission eigenvalues and thermoacoustic tomography}, Inverse Problems, \textbf{29} (2013), 104016.
\bibitem{HK} {\sc G. Hu and Y. Kian}, {\em Uniqueness and stability for the recovery of a time-dependent source and initial conditions in elastodynamics}, Inverse Probl. Imaging, \textbf{14} (2020), 463-487.

\bibitem{JW} {\sc X. Jin and L. V. Wang}, {\em Thermoacoustic tomography with correction for acoustic speed variations},
Phys. Med. Biol., \textbf{51} (2006), 6437.
\bibitem{KSXY}{\sc Y. Kian, E. Soccorsi, Q. Xue, M. Yamamoto}, {\em Identification of time-varying source term in time-fractional diffusion equations},  Communication in Mathematical Sciences, \textbf{20} (2022), 53-84.
\bibitem{KM}{\sc C. Knox and A. Moradifam}, {\em Determining both the source of a wave and its
speed in a medium from boundary measurement}, Inverse Problems, \textbf{36} (2020), 025002.
\bibitem{KRK}{\sc R. A. Kruger, D. R. Reinecke and G. A. Kruger},  {\em Thermoacoustic computed tomographytechnical
construction}, Med. Phys., \textbf{26} (1999), 1832-1837.
\bibitem{LU}{\sc  H. Liu and G. Uhlmann}, {\em Determining both sound speed and internal source in thermo- and
photo-acoustic tomography}, Inverse Problems, \textbf{31} (2015), 105005.
\bibitem{NN}{\sc H-M. Nguyen and Q-H. Nguyen},{\em The Weyl law of transmission eigenvalues and
the completeness of generalized transmission eigenfunctions}, Journal of Functional Analysis, \textbf{281} (2021) 109-146. 
 \bibitem{PV}{\sc V. Petkov and G. Vodev}, {\em Asymptotics of the number of the interior transmission eigenvalues},
J. Spectr. Theory, \textbf{7} (1) (2017) 1–31. 
\bibitem{R1}{\sc L. Robbiano}, {\em Spectral analysis of the interior transmission eigenvalue problem}, Inverse Probl., \textbf{29} (10) (2013), 104001. 
 \bibitem{SU1}{\sc P. Stefanov  and G. Uhlmann}, {\em Thermoacoustic tomography with variable sound speed}, Inverse
Problems, \textbf{25} (2009), 075011.
\bibitem{SU2}{\sc P. Stefanov  and G. Uhlmann}, {\em  Thermoacoustic tomography arising in brain imaging}, Inverse
Problems, \textbf{27} (2011), 045004.
\bibitem{SU3}{\sc P. Stefanov  and G. Uhlmann}, {\em Instability of the linearized problem in multiwave tomography of
recovery both the source and the speed}, Inverse Problems Imaging, \textbf{7} (2013), 1367-1377.
\bibitem{SU4}{\sc P. Stefanov  and G. Uhlmann}, {\em Recovery of a source term or a speed with one measurement and
applications}, Trans. Am. Math. Soc., \textbf{365} (2013), 5737-5758.
\bibitem{T}{\sc A. C. Tam}, {\em Applications of photoacoustic sensing techniques}, Rev. Mod. Phys., \textbf{58} (1986), 381-431.
\bibitem{Ta}{\sc D. Tataru}, {\em  Unique continuation for solutions to PDE; between H\"ormander's theorem and Holmgren's theorem}, Commun. Partial Diff. Eqns., \textbf{20} (1995), 855-884.
\bibitem{Va}{\sc B. R. Vainberg} {\em Asymptotic Methods in Equations of Mathematical Physics}, Gordon and Breach, New York, 1988.
\bibitem{Vo1}{\sc  G. Vodev}, {\em Sharp Bounds on the Number of Scattering Poles for Perturbations of the Laplacian}, Commun. Math. Phys., \textbf{146} (1992), 205-216.
\bibitem{Vo2}{\sc  G. Vodev}, {\em On the uniform decay of the local energy}, Serdica Math. J., \textbf{25} (1999), 191-206.
\bibitem{Vo3}{\sc  G. Vodev}, {\em High-frequency approximation of the interior Dirichlet-to-Neumann map and applications to the transmission eigenvalues}, Anal. PDE, \textbf{11} (1) (2018), 213-236.
\bibitem{W}{\sc L. Wang} {\em Photoacoustic imaging and spectroscopy Optical Science and Engineering},  2009,  Taylor and Francis, London.
\end{thebibliography}
\end{document}